\documentclass{amsart}
\usepackage{amssymb}
\usepackage{amsfonts}

\newtheorem{theorem}{Theorem}
\theoremstyle{plain}

\newtheorem{corollary}{Corollary}

\newtheorem{definition}{Definition}

\newtheorem{lemma}{Lemma}

\newtheorem{remark}{Remark}

\numberwithin{equation}{section}
\input{tcilatex}

\begin{document}
\title[Hermite-Hadamard type inequalities for h-convex function]{Some new
inequalities of Hermite-Hadamard type for $h-$convex functions on the
co-ordinates via fractional integrals.}
\author{Erhan SET$^{\blacktriangle }$}
\address{$^{\blacktriangle }$Department of Mathematics, \ Faculty of Science
and Arts, Ordu University, Ordu-TURKEY}
\email{erhanset@yahoo.com}
\author{M. Zeki Sar\i kaya$^{\blacksquare }$}
\address{$^{\blacksquare }$Department of Mathematics, \ Faculty of Science
and Arts, D\"{u}zce University, D\"{u}zce-TURKEY}
\email{sarikayamz@gmail.com}
\author{Hatice \"{O}g\"{u}lm\"{u}\c{s}$^{\clubsuit }$}
\address{$^{\clubsuit }$Department of Mathematics, \ Faculty of Science and
Arts, D\"{u}zce University, D\"{u}zce-TURKEY}
\email{hat\i ceogulmus@hotmail.com}

\begin{abstract}
By making use of the identity obtained by Sar\i kaya, some new
Hermite-Hadamard type inequalities for h-convex functions on the
co-ordinates via fractional integrals are established. Our results have some
relationships with the results of Sar\i kaya(\cite{M.Z.S})
\end{abstract}

\maketitle

\section{INTRODUCTION}

Let $f:I\subseteq 
\mathbb{R}
\rightarrow \mathbb{R}$ be a convex function defined on the interval $I$ of
real numbers and $a,b\in I$ with $a<b$ , then%
\begin{equation*}
f(\frac{a+b}{2})\leq \frac{1}{b-a}\int_{a}^{b}f(x)dx\leq \frac{f(a)+f(b)}{2}.
\end{equation*}%
holds, the double inequality is well known in the literature as
Hermite-Hadamard inequality \cite{DraPe}:

Both inequalities hold in the reversed direction if $f$ is concave. For
recent results, generalizations and new inequalities related to the
Hermite-Hadamard inequality see ( \cite{Dra}, \cite{P.P},\cite{SSY}, \cite%
{S.O.D})

The classical Hermite- Hadamard inequality provides estimates of the mean
value of a continuous convex function $f:\left[ a,b\right] \rightarrow 
\mathbb{R}$

Let us now consider a bidemensional interval $\Delta :=\left[ a,b\right]
\times \left[ c,d\right] $ in $%
\mathbb{R}
^{2}$ with $a<b$ and $c<d$.\ A mapping $f:\Delta \rightarrow 
\mathbb{R}
$ is said to be convex on if the following inequality:%
\begin{equation*}
f(tx+(1-t)z,ty+(1-t)w)\leq tf(x,y)+(1-t)f(z,w)
\end{equation*}%
holds, for all $(x,y),(z,w)\in \Delta $ and $t\in \left[ 0,1\right] $ . If
the inequality reversed then $f$ is said to be concave on $\Delta $.

A function $f:\Delta \rightarrow 
\mathbb{R}
$ is said to be on the co-ordinates on $\Delta $ if the partial mappings $%
f_{y}:[a,b]\rightarrow 
\mathbb{R}
,f_{y}(u)=f(u,y)$ and $f_{x}:[c,d]\rightarrow 
\mathbb{R}
,f_{x}(v)=f(x,v)$ are convex where defined for all $x\in \left[ a,b\right] $
and $y\in \left[ c,d\right] .$ (see \cite{Dragomir})

A formal definition for co-ordinated convex function may be stated as
follows:

\begin{definition}
\label{d1}A function $f:\Delta \rightarrow 
\mathbb{R}
$ will be called co-ordinated convex on $\Delta $ , for all $t,k\in \left[
0,1\right] $ and $\left( x,u\right) ,\left( y,w\right) \in \Delta $ ,if the
following inequality holds: 
\begin{eqnarray*}
&&f\left( tx+\left( 1-t\right) y,ku+\left( 1-k\right) w\right) \\
&\leq &tkf\left( x,u\right) +k\left( 1-t\right) f\left( y,u\right) +t\left(
1-k\right) f\left( x,w\right) +\left( 1-t\right) \left( 1-k\right) f\left(
y,w\right)
\end{eqnarray*}
\end{definition}

Clearly, every convex mapping is convex on the co-ordinates, but the
converse is not generally true (\cite{Dragomir}). Some interesting and
important inequalities for convex functions on the co-ordinates can be found
in (\cite{9}, \cite{10}, \cite{SSO})

In \cite{Dralamori}, Alomari and Darus established the following definition
of $s$--convex function in the second sense on co--ordinates.

\begin{definition}
\label{d2}Consider the bidimensional interval $\Delta :=\left[ a,b\right]
\times \left[ c,d\right] $ in $\left[ 0,\infty \right) ^{2}$ with $a<b$ and $%
c<d.$ The mapping $f:\Delta \rightarrow 
\mathbb{R}
$ is $s-$convex on $\Delta $ if%
\begin{equation*}
f\left( \lambda x+\left( 1-\lambda \right) z,\lambda y+\left( 1-\lambda
\right) w\right) \leq \lambda ^{s}f\left( x,y\right) +\left( 1-\lambda
\right) ^{s}f\left( z,w\right)
\end{equation*}%
holds for all $(x,y),(z,w)\in \Delta $ with $\lambda \in \left[ 0,1\right] ,$
and for some fixed $s\in \left( 0,1\right] .$
\end{definition}

A function $f:\Delta \rightarrow 
\mathbb{R}
$ is s-convex on $\Delta $\ is called co-ordinated s-convex on $\Delta $ if
the partial mappings $f_{y}:[a,b]\rightarrow 
\mathbb{R}
,$ $f_{y}(u)=f(u,y)$ and $f_{x}:[c,d]\rightarrow 
\mathbb{R}
,$ $f_{x}(v)=f(x,v)$ are s-convex for all $x\in \left[ a,b\right] $ and $%
y\in \left[ c,d\right] $ with some fixed $s\in \left( 0,1\right] .$

In \cite{Latif}, Latif and Alamori give the notion of h-convexity of a
function $f$ on a rectangle from the plane $%
\mathbb{R}
^{2}$ and h-convexity on the co-ordinates on a rectangle from the plane $%
\mathbb{R}
^{2}$ as follows

\begin{definition}
\label{d3}Let us consider a bidimensional interval $\Delta :=\left[ a,b%
\right] \times \left[ c,d\right] $ in $%
\mathbb{R}
^{2}$ with $a<b$ and $c<d.$ Let $h:J\subseteq 
\mathbb{R}
\rightarrow 
\mathbb{R}
$ be a positivie function. A mapping $f:\Delta :=\left[ a,b\right] \times %
\left[ c,d\right] \rightarrow 
\mathbb{R}
$ is said to be h-convex on $\Delta ,$ if f is non-negative and if the
following inequality:%
\begin{equation*}
f\left( \lambda x+\left( 1-\lambda \right) z,\lambda y+\left( 1-\lambda
\right) w\right) \leq h\left( \lambda \right) f\left( x,y\right) +h\left(
1-\lambda \right) f\left( z,w\right)
\end{equation*}%
holds, for all $(x,y),(z,w)\in \Delta $ with $\lambda \in \left( 0,1\right)
. $ Let us denote this class of functions by $SX\left( h,\Delta \right) .$
The function\ f is said to be h-concave if the inequality reversed. We denot
this class of functions by $SV\left( h,\Delta \right) .$
\end{definition}

A function $f:\Delta \rightarrow 
\mathbb{R}
$ is said to be $h$-convex on the co-ordinates on $\Delta $\ if the partial
mappings $f_{y}:[a,b]\rightarrow 
\mathbb{R}
,$ $f_{y}(u)=f(u,y)$ and $f_{x}:[c,d]\rightarrow 
\mathbb{R}
,$ $f_{x}(v)=f(x,v)$ are $h$-convex where defined for all $x\in \left[ a,b%
\right] $ and $y\in \left[ c,d\right] $. A formal definition of $h$-convex
functions may also be stated as follows:

\begin{definition}
\label{d4} \cite{Latif} A function $f:\Delta \rightarrow 
\mathbb{R}
$ is said to be h-convex on the co-ordinates on $\Delta $ , if the following
inequality:%
\begin{align}
&  \label{1} \\
& f\left( tx+\left( 1-t\right) y,ku+\left( 1-k\right) w\right)   \notag \\
& \leq h\left( t\right) h\left( k\right) f\left( x,u\right) +h\left(
k\right) h\left( 1-t\right) f\left( y,u\right)   \notag \\
& +h\left( t\right) h\left( 1-k\right) f\left( x,w\right) +h\left(
1-t\right) h\left( 1-k\right) f\left( y,w\right)   \notag
\end{align}%
holds for all $t,k\in \left[ 0,1\right] $ and $\left( x,u\right) ,\left(
x,w\right) ,\left( y,u\right) ,\left( y,w\right) \in \Delta .$
\end{definition}

Obviously, if $h(\alpha )=\alpha $, then all the non-negative convex
(concave) functions on $\Delta $\ belong to the class $SX(h,\Delta )$ ($%
SV\left( h,\Delta \right) $) and if $h(\alpha )=\alpha ^{s}$, where $s\in
(0,1)$, then the class of $s$-convex on \ $\Delta $ belong to the class $%
SX(h,\Delta )$. Similarly we can say that if $h(\alpha )=\alpha ,$ then the
class of non-negative convex (concave) functions on the co-ordinates on $%
\Delta $\ is contained in the class of $h$-convex (concave) functions on the
co-ordinates on $\Delta $\ and if $h(\alpha )=\alpha ^{s},$where $s\in (0,1)$%
, then the class of $s$-convex functions on the co-ordinates on $\Delta $\
is contained in the class of $h$-convex functions on the co-ordinates on $%
\Delta .$

In the following we will give some necessary definitions which are used
further in this paper. \ More details, one can consult [\cite{7}, \cite{8}, 
\cite{13}]

\begin{definition}
\label{d5}Let $f\in L_{1}\left[ a,b\right] $. The Riemann-Liouville
integrals $J_{a^{+}}^{\alpha }f$ and $J_{b^{-}}^{\alpha }f$ order $\alpha >0$
and $a\geq 0$ are defined by%
\begin{equation*}
J_{a^{+}}^{\alpha }f\left( x\right) =\frac{1}{\Gamma \left( \alpha \right) }%
\int_{a}^{x}\left( x-t\right) ^{\alpha -1}f\left( t\right) dt,\text{ \ }x>a
\end{equation*}%
and%
\begin{equation*}
J_{b^{-}}^{\alpha }f\left( x\right) =\frac{1}{\Gamma \left( \alpha \right) }%
\int_{x}^{b}\left( t-x\right) ^{\alpha -1}f\left( t\right) dt,\text{ \ }x<b
\end{equation*}%
respectively. Here $\Gamma \left( \alpha \right) $ is the Gamma function and 
$J_{a^{+}}^{0}f\left( x\right) =J_{b^{-}}^{0}f\left( x\right) =f\left(
x\right) .$
\end{definition}

\begin{definition}
\label{d6}Let $f\in L_{1}\left( \left[ a,b\right] \times \left[ c,d\right]
\right) $. the Riemann-Liouville integrals $J_{a^{+},c^{+}}^{\alpha ,\beta },
$ $J_{a^{+},d^{-}}^{\alpha ,\beta },$ $J_{b^{-},c^{+}}^{\alpha ,\beta }$ and 
$J_{b^{-},d^{-}}^{\alpha ,\beta }$ of order $\alpha ,\beta >0$ with $a,c\geq
0$ are defined by%
\begin{eqnarray*}
&&J_{a^{+},c^{+}}^{\alpha ,\beta }f\left( x,y\right)  \\
&& \\
&=&\frac{1}{\Gamma \left( \alpha \right) \Gamma \left( \beta \right) }%
\int_{a}^{x}\int_{c}^{y}\left( x-t\right) ^{\alpha -1}\left( y-s\right)
^{\beta -1}f\left( t,s\right) dsdt,\text{\ }x>a,y>c
\end{eqnarray*}%
\begin{eqnarray*}
&&J_{a^{+},d^{-}}^{\alpha ,\beta }f\left( x,y\right)  \\
&=&\frac{1}{\Gamma \left( \alpha \right) \Gamma \left( \beta \right) }%
\int_{a}^{x}\int_{y}^{d}\left( x-t\right) ^{\alpha -1}\left( s-y\right)
^{\beta -1}f\left( t,s\right) dsdt,\text{ }x>a,y<d
\end{eqnarray*}%
\begin{eqnarray*}
&&J_{b^{-},c^{+}}^{\alpha ,\beta }f\left( x,y\right)  \\
&=&\frac{1}{\Gamma \left( \alpha \right) \Gamma \left( \beta \right) }%
\int_{x}^{b}\int_{c}^{y}\left( t-x\right) ^{\alpha -1}\left( y-s\right)
^{\beta -1}f\left( t,s\right) dsdt,\text{ }x<b,y>c
\end{eqnarray*}%
\begin{eqnarray*}
&&J_{b^{-},d^{-}}^{\alpha ,\beta }f\left( x,y\right)  \\
&=&\frac{1}{\Gamma \left( \alpha \right) \Gamma \left( \beta \right) }%
\int_{x}^{b}\int_{y}^{d}\left( t-x\right) ^{\alpha -1}\left( s-y\right)
^{\beta -1}f\left( t,s\right) dsdt,\text{ }x<b,y<d
\end{eqnarray*}%
respectively. Here, $\Gamma $ is the Gama function,%
\begin{equation*}
J_{a^{+},c^{+}}^{0,0}f\left( x,y\right) =J_{a^{+},d^{-}}^{0,0}f\left(
x,y\right) =J_{b^{-},c^{+}}^{0,0}f\left( x,y\right)
=J_{b^{-},d^{-}}^{0,0}f\left( x,y\right) =f\left( x,y\right) 
\end{equation*}%
and%
\begin{equation*}
J_{a^{+},c^{+}}^{1,1}f\left( x,y\right) =\int_{a}^{x}\int_{c}^{y}f\left(
t,s\right) dsdt.
\end{equation*}
\end{definition}

For some recent results connected with fractional integral inequalities see (%
\cite{2}, \cite{3}, \cite{MzsHy}, \cite{Set}).

In (\cite{M.Z.S}), Sar\i kaya establish the following inequalities of
Hadamard's type for co-ordinated convex mapping on a rectangle from the
plane $%
\mathbb{R}
^{2}$:

\begin{theorem}
\label{t1}Let $f:\Delta \subset 
\mathbb{R}
^{2}\rightarrow 
\mathbb{R}
$ be co-ordinated convex on $\Delta :=\left[ a,b\right] \times \left[ c,d%
\right] $ in $%
\mathbb{R}
^{2}$ with $0\leq a<b$, $0\leq c<d$ and $f\in L_{1}\left( \Delta \right) .$
Then one has the inequalities:%
\begin{eqnarray}
&&  \label{2} \\
&&f\left( \frac{a+b}{2},\frac{c+d}{2}\right)  \notag \\
&&  \notag \\
&\leq &\frac{\Gamma \left( \alpha +1\right) \Gamma \left( \beta +1\right) }{%
4\left( b-a\right) ^{\alpha }\left( d-c\right) ^{\beta }}  \notag \\
&&  \notag \\
&&\left[ J_{a^{+},c^{+}}^{\alpha ,\beta }f\left( b,d\right)
+J_{a^{+},d^{-}}^{\alpha ,\beta }f\left( b,c\right) +J_{b^{-},c^{+}}^{\alpha
,\beta }f\left( a,d\right) +J_{b^{-},d^{-}}^{\alpha ,\beta }f\left(
a,c\right) \right]  \notag \\
&&  \notag \\
&\leq &\frac{f\left( a,c\right) +f\left( a,d\right) +f\left( b,c\right)
+f\left( b,d\right) }{4}.  \notag
\end{eqnarray}
\end{theorem}

\begin{theorem}
\label{t2}Let $f:\Delta \subset 
\mathbb{R}
^{2}\rightarrow 
\mathbb{R}
$ be a partial differentiable mapping on $\Delta :=\left[ a,b\right] \times %
\left[ c,d\right] $ in $%
\mathbb{R}
^{2}$ with $0\leq a<b$, $0\leq c<d$. If $\left\vert \frac{\partial ^{2}f}{%
\partial t\partial k}\right\vert $ is a convex function on the co-ordinates
on $\Delta $, then one has the inequalities:%
\begin{eqnarray}
&&  \label{3} \\
&&\left\vert \frac{f\left( a,c\right) +f\left( a,d\right) +f\left(
b,c\right) +f\left( b,d\right) }{4}\right.  \notag \\
&&  \notag \\
&&\left. +\frac{\Gamma \left( \alpha +1\right) \Gamma \left( \beta +1\right) 
}{4\left( b-a\right) ^{\alpha }\left( d-c\right) ^{\beta }}\right.  \notag \\
&&  \notag \\
&&\left. \times \left[ J_{a^{+},c^{+}}^{\alpha ,\beta }f\left( b,d\right)
+J_{a^{+},d^{-}}^{\alpha ,\beta }f\left( b,c\right) +J_{b^{-},c^{+}}^{\alpha
,\beta }f\left( a,d\right) +J_{b^{-},d^{-}}^{\alpha ,\beta }f\left(
a,c\right) \right] -A\right\vert  \notag \\
&&  \notag \\
&\leq &\frac{\left( b-a\right) \left( d-c\right) }{4\left( \alpha +1\right)
\left( \beta +1\right) }  \notag \\
&&  \notag \\
&&\times \left( \left\vert \frac{\partial ^{2}f}{\partial k\partial t}\left(
a,c\right) \right\vert +\left\vert \frac{\partial ^{2}f}{\partial k\partial t%
}\left( a,d\right) \right\vert +\left\vert \frac{\partial ^{2}f}{\partial
k\partial t}\left( b,c\right) \right\vert +\left\vert \frac{\partial ^{2}f}{%
\partial k\partial t}\left( b,d\right) \right\vert \right)  \notag
\end{eqnarray}%
where%
\begin{eqnarray*}
A &=&\frac{\Gamma \left( \beta +1\right) }{4\left( d-c\right) ^{\beta }}%
\left[ J_{c^{+}}^{\beta }f\left( a,d\right) +J_{c^{+}}^{\beta }f\left(
b,d\right) +J_{d^{-}}^{\beta }f\left( a,c\right) +J_{d^{-}}^{\beta }f\left(
b,c\right) \right] \\
&& \\
&&+\frac{\Gamma \left( \alpha +1\right) }{4\left( b-a\right) ^{\alpha }}%
\left[ J_{a^{+}}^{\alpha }f\left( b,c\right) +J_{a^{+}}^{\alpha }f\left(
b,d\right) +J_{b^{-}}^{\alpha }f\left( a,c\right) +J_{b^{-}}^{\alpha
}f\left( a,d\right) \right] .
\end{eqnarray*}
\end{theorem}

\begin{theorem}
\label{t3}Let $f:\Delta \subset 
\mathbb{R}
^{2}\rightarrow 
\mathbb{R}
$ be a partial differentiable mapping on $\Delta :=\left[ a,b\right] \times %
\left[ c,d\right] $ in $%
\mathbb{R}
^{2}$ with $0\leq a<b$, $0\leq c<d$. If $\left\vert \frac{\partial ^{2}f}{%
\partial t\partial k}\right\vert ^{q},q>1,$ is a convex function on the
co-ordinates on $\Delta $, then one has the inequalities:%
\begin{eqnarray}
&&  \label{4} \\
&&\left\vert \frac{f\left( a,c\right) +f\left( a,d\right) +f\left(
b,c\right) +f\left( b,d\right) }{4}\right.  \notag \\
&&  \notag \\
&&+\left\{ \frac{\Gamma \left( \alpha +1\right) \Gamma \left( \beta
+1\right) }{4\left( b-a\right) ^{\alpha }\left( d-c\right) ^{\beta }}\right.
\notag \\
&&  \notag \\
&&\left. \times \left[ J_{a^{+},c^{+}}^{\alpha ,\beta }f\left( b,d\right)
+J_{a^{+},d^{-}}^{\alpha ,\beta }f\left( b,c\right) +J_{b^{-},c^{+}}^{\alpha
,\beta }f\left( a,d\right) +J_{b^{-},d^{-}}^{\alpha ,\beta }f\left(
a,c\right) \right] \right\} -A  \notag \\
&&  \notag \\
&\leq &\frac{\left( b-a\right) \left( d-c\right) }{\left[ \left( \alpha
p+1\right) \left( \beta p+1\right) \right] ^{\frac{1}{p}}}\left( \frac{1}{4}%
\right) ^{\frac{1}{q}}  \notag \\
&&\times \left( \left\vert \frac{\partial ^{2}f}{\partial k\partial t}\left(
a,c\right) \right\vert ^{q}+\left\vert \frac{\partial ^{2}f}{\partial
k\partial t}\left( a,d\right) \right\vert ^{q}+\left\vert \frac{\partial
^{2}f}{\partial k\partial t}\left( b,c\right) \right\vert ^{q}+\left\vert 
\frac{\partial ^{2}f}{\partial k\partial t}\left( b,d\right) \right\vert
^{q}\right) ^{\frac{1}{q}}  \notag
\end{eqnarray}%
where%
\begin{eqnarray*}
A &=&\frac{\Gamma \left( \beta +1\right) }{4\left( d-c\right) ^{\beta }}%
\left[ J_{c^{+}}^{\beta }f\left( a,d\right) +J_{c^{+}}^{\beta }f\left(
b,d\right) +J_{d^{-}}^{\beta }f\left( a,c\right) +J_{d^{-}}^{\beta }f\left(
b,c\right) \right] \\
&& \\
&&+\frac{\Gamma \left( \alpha +1\right) }{4\left( b-a\right) ^{\alpha }}%
\left[ J_{a^{+}}^{\alpha }f\left( b,c\right) +J_{a^{+}}^{\alpha }f\left(
b,d\right) +J_{b^{-}}^{\alpha }f\left( a,c\right) +J_{b^{-}}^{\alpha
}f\left( a,d\right) \right]
\end{eqnarray*}%
and $\frac{1}{p}+\frac{1}{q}=1.$
\end{theorem}

In order to prove our main results we need the following lemma (see \cite%
{M.Z.S}).

\begin{lemma}
\label{l1}Let $f:\Delta \subset 
\mathbb{R}
^{2}\rightarrow 
\mathbb{R}
$ be a partial differentiable mapping on $\Delta :=\left[ a,b\right] \times %
\left[ c,d\right] $ in $%
\mathbb{R}
^{2}$ with $0\leq a<b$, $0\leq c<d$. If $\left. \frac{\partial ^{2}f}{%
\partial t\partial k}\in L\left( \Delta \right) \right. ,$ then the
following equality holds:%
\begin{eqnarray}
&&  \label{5} \\
&&\frac{f\left( a,c\right) +f\left( a,d\right) +f\left( b,c\right) +f\left(
b,d\right) }{4}  \notag \\
&&  \notag \\
&&+\left\{ \frac{\Gamma \left( \alpha +1\right) \Gamma \left( \beta
+1\right) }{4\left( b-a\right) ^{\alpha }\left( d-c\right) ^{\beta }}\right. 
\notag \\
&&  \notag \\
&&\times \left. \left[ J_{a^{+},c^{+}}^{\alpha ,\beta }f\left( b,d\right)
+J_{a^{+},d^{-}}^{\alpha ,\beta }f\left( b,c\right) +J_{b^{-},c^{+}}^{\alpha
,\beta }f\left( a,d\right) +J_{b^{-},d^{-}}^{\alpha ,\beta }f\left(
a,c\right) \right] \right\}   \notag \\
&&  \notag \\
&&-\frac{\Gamma \left( \beta +1\right) }{4\left( d-c\right) ^{\beta }}\left[
J_{c^{+}}^{\beta }f\left( a,d\right) +J_{c^{+}}^{\beta }f\left( b,d\right)
+J_{d^{-}}^{\beta }f\left( a,c\right) +J_{d^{-}}^{\beta }f\left( b,c\right) %
\right]   \notag \\
&&  \notag \\
&&-\frac{\Gamma \left( \alpha +1\right) }{4\left( b-a\right) ^{\alpha }}%
\left[ J_{a^{+}}^{\alpha }f\left( b,c\right) +J_{a^{+}}^{\alpha }f\left(
b,d\right) +J_{b^{-}}^{\alpha }f\left( a,c\right) +J_{b^{-}}^{\alpha
}f\left( a,d\right) \right]   \notag
\end{eqnarray}%
\begin{eqnarray*}
&=&\frac{\left( b-a\right) \left( d-c\right) }{4}\left\{
\int_{0}^{1}\int_{0}^{1}t^{\alpha }k^{\beta }\frac{\partial ^{2}f}{\partial
t\partial k}\left( ta+\left( 1-t\right) b,kc+\left( 1-k\right) d\right)
dkdt\right.  \\
&& \\
&&-\int_{0}^{1}\int_{0}^{1}\left( 1-t\right) ^{\alpha }k^{\beta }\frac{%
\partial ^{2}f}{\partial t\partial k}\left( ta+\left( 1-t\right) b,kc+\left(
1-k\right) d\right) dkdt \\
&& \\
&&-\int_{0}^{1}\int_{0}^{1}t^{\alpha }\left( 1-k\right) ^{\beta }\frac{%
\partial ^{2}f}{\partial t\partial k}\left( ta+\left( 1-t\right) b,kc+\left(
1-k\right) d\right) dkdt \\
&& \\
&&\left. -\int_{0}^{1}\int_{0}^{1}\left( 1-t\right) ^{\alpha }\left(
1-k\right) ^{\beta }\frac{\partial ^{2}f}{\partial t\partial k}\left(
ta+\left( 1-t\right) b,kc+\left( 1-k\right) d\right) dkdt\right\} .
\end{eqnarray*}
\end{lemma}

\section{MAIN RESULTS}

\begin{theorem}
\label{t4}Let $f:\Delta \subset 
\mathbb{R}
^{2}\rightarrow 
\mathbb{R}
$ be h-convex function on the co-ordinates on $\Delta :=\left[ a,b\right]
\times \left[ c,d\right] $ in $%
\mathbb{R}
^{2}$ and $f\in L_{2}\left( \Delta \right) .$ The one has the inequalities:%
\begin{eqnarray}
&&  \label{6} \\
&&f\left( \frac{a+b}{2},\frac{c+d}{2}\right)  \notag \\
&&  \notag \\
&\leq &\left[ h\left( \frac{1}{2}\right) \right] ^{2}\frac{\Gamma \left(
\alpha +1\right) \Gamma \left( \beta +1\right) }{\left( b-a\right) ^{\alpha
}\left( d-c\right) ^{\beta }}  \notag \\
&&  \notag \\
&&\times \left[ J_{a^{+},c^{+}}^{\alpha ,\beta }f\left( b,d\right)
+J_{a^{+},d^{-}}^{\alpha ,\beta }f\left( b,c\right) +J_{b^{-},c^{+}}^{\alpha
,\beta }f\left( a,d\right) +J_{b^{-},d^{-}}^{\alpha ,\beta }f\left(
a,c\right) \right]  \notag \\
&&  \notag \\
&\leq &\left[ h\left( \frac{1}{2}\right) \right] ^{2}\alpha \beta \left[
f\left( a,c\right) +f\left( a,d\right) +f\left( b,c\right) +f\left(
b,d\right) \right]  \notag \\
&&  \notag \\
&&\times \left[ \int_{0}^{1}\int_{0}^{1}t^{\alpha -1}k^{\alpha -1}\right. 
\notag \\
&&  \notag \\
&&\left. \times \left[ h\left( t\right) h\left( k\right) +h\left( t\right)
h\left( 1-k\right) +h\left( 1-t\right) h\left( k\right) +h\left( 1-t\right)
h\left( 1-k\right) \right] dkdt\right] .  \notag
\end{eqnarray}
\end{theorem}

\begin{proof}
According to (\ref{1}) with $x=t_{1}a+\left( 1-t_{1}\right) b,y=\left(
1-t_{1}\right) a+t_{1}b,u=k_{1}c+\left( 1-k_{1}\right) d,w=\left(
1-k_{1}\right) c+k_{1}d$ and $t=k=\frac{1}{2}$, we find that%
\begin{eqnarray}
&&  \label{7} \\
&&f\left( \frac{a+b}{2},\frac{c+d}{2}\right)  \notag \\
&&  \notag \\
&\leq &\left[ h\left( \frac{1}{2}\right) \right] ^{2}  \notag \\
&&  \notag \\
&&\times \left[ f\left( t_{1}a+\left( 1-t_{1}\right) b,k_{1}c+\left(
1-k_{1}\right) d\right) +f\left( t_{1}a+\left( 1-t_{1}\right) b,\left(
1-k_{1}\right) c+k_{1}d\right) \right.  \notag \\
&&  \notag \\
&&\left. +f\left( \left( 1-t_{1}\right) a+t_{1}b,k_{1}c+\left(
1-k_{1}\right) d\right) +f\left( \left( 1-t_{1}\right) a+t_{1}b,\left(
1-k_{1}\right) c+k_{1}d\right) \right] .  \notag
\end{eqnarray}

Thus, multiplying both sides of (\ref{7}) by $t_{1}^{\alpha -1}k_{1}^{\beta
-1},$ then by integrating with respect to $\left( t_{1},k_{1}\right) $ on $%
\left[ 0,1\right] \times \left[ 0,1\right] ,$ we obtain%
\begin{eqnarray*}
&& \\
&&\frac{1}{\alpha \beta }f\left( \frac{a+b}{2},\frac{c+d}{2}\right)  \\
&& \\
&\leq &\left[ h\left( \frac{1}{2}\right) \right] ^{2} \\
&& \\
&&\times \left[ 
\begin{array}{c}
\int_{0}^{1}\int_{0}^{1}t_{1}^{\alpha -1}k_{1}^{\alpha -1} \\ 
\times \left[ f\left( t_{1}a+\left( 1-t_{1}\right) b,k_{1}c+\left(
1-k_{1}\right) d\right) \right.  \\ 
+f\left( t_{1}a+\left( 1-t_{1}\right) b,\left( 1-k_{1}\right)
c+k_{1}d\right)  \\ 
+f\left( \left( 1-t_{1}\right) a+t_{1}b,k_{1}c+\left( 1-k_{1}\right)
d\right)  \\ 
\left. +f\left( \left( 1-t_{1}\right) a+t_{1}b,\left( 1-k_{1}\right)
c+k_{1}d\right) \right] dk_{1}ds_{1}.%
\end{array}%
\right. 
\end{eqnarray*}

Using the change of the varible, we get%
\begin{eqnarray*}
&&f\left( \frac{a+b}{2},\frac{c+d}{2}\right) \\
&& \\
&\leq &\left[ h\left( \frac{1}{2}\right) \right] ^{2}\frac{\alpha \beta }{%
\left( b-a\right) ^{\alpha }\left( d-c\right) ^{\beta }}\left\{
\int_{a}^{b}\int_{c}^{d}\left( b-x\right) ^{\alpha -1}\left( d-y\right)
^{\beta -1}f\left( x,y\right) dydx\right. \\
&& \\
&&+\int_{a}^{b}\int_{c}^{d}\left( b-x\right) ^{\alpha -1}\left( y-c\right)
^{\beta -1}f\left( x,y\right) dydx \\
&& \\
&&+\int_{a}^{b}\int_{c}^{d}\left( x-a\right) ^{\alpha -1}\left( d-y\right)
^{\beta -1}f\left( x,y\right) dydx \\
&& \\
&&\left. +\int_{a}^{b}\int_{c}^{d}\left( x-a\right) ^{\alpha -1}\left(
y-c\right) ^{\beta -1}f\left( x,y\right) dydx\right\}
\end{eqnarray*}%
which the first inequality is proved.

For the proof of second inequality (\ref{6}), we first note that if f is a $h
$-convex function on $\Delta $, then, by using (\ref{1}) with $%
x=a,y=b,u=c,w=d$, it yields%
\begin{eqnarray*}
&&f\left( ta+\left( 1-t\right) b,kc+\left( 1-k\right) d\right)  \\
&\leq &h\left( t\right) h\left( k\right) f\left( a,c\right) +h\left(
t\right) h\left( 1-k\right) f\left( a,d\right)  \\
&&+h\left( 1-t\right) h\left( k\right) f\left( b,c\right) +h\left(
1-t\right) h\left( 1-k\right) f\left( b,d\right) 
\end{eqnarray*}%
\begin{eqnarray*}
&&f\left( \left( 1-t\right) a+tb,kc+\left( 1-k\right) d\right)  \\
&\leq &h\left( 1-t\right) h\left( k\right) f\left( a,c\right) +h\left(
1-t\right) h\left( 1-k\right) f\left( a,d\right)  \\
&&+h\left( t\right) h\left( k\right) f\left( b,c\right) +h\left( t\right)
h\left( 1-k\right) f\left( b,d\right) 
\end{eqnarray*}%
\begin{eqnarray*}
&&f\left( ta+\left( 1-t\right) b,\left( 1-k\right) c+kd\right)  \\
&\leq &h\left( t\right) h\left( 1-k\right) f\left( a,c\right) +h\left(
t\right) h\left( k\right) f\left( a,d\right)  \\
&&+h\left( 1-t\right) h\left( 1-k\right) f\left( b,c\right) +h\left(
1-t\right) h\left( k\right) f\left( b,d\right) 
\end{eqnarray*}%
\begin{eqnarray*}
&&f\left( \left( 1-t\right) a+tb,\left( 1-k\right) c+kd\right)  \\
&\leq &h\left( 1-t\right) h\left( 1-k\right) f\left( a,c\right) +h\left(
1-t\right) h\left( k\right) f\left( a,d\right)  \\
&&+h\left( t\right) h\left( 1-k\right) f\left( b,c\right) +h\left( t\right)
h\left( k\right) f\left( b,d\right) .
\end{eqnarray*}

By adding these inequalities we have%
\begin{eqnarray}
&&  \label{8} \\
&&f\left( ta+\left( 1-t\right) b,kc+\left( 1-k\right) d\right) +f\left(
\left( 1-t\right) a+tb,kc+\left( 1-k\right) d\right)   \notag \\
&&  \notag \\
&&+f\left( ta+\left( 1-t\right) b,\left( 1-k\right) c+kd\right) +f\left(
\left( 1-t\right) a+tb,\left( 1-k\right) c+kd\right)   \notag \\
&&  \notag \\
&\leq &\left[ f\left( a,c\right) +f\left( b,c\right) +f\left( a,d\right)
+f\left( b,d\right) \right]   \notag \\
&&  \notag \\
&&\times \left[ h\left( t\right) h\left( k\right) +h\left( t\right) h\left(
1-k\right) +h\left( 1-t\right) h\left( k\right) +h\left( 1-t\right) h\left(
1-k\right) \right] .  \notag
\end{eqnarray}%
Then, multiplying both sides of (\ref{8}) by $t^{\alpha -1}k^{\beta -1}$ and
integrating with respect to $\left( t,k\right) $ over $\left[ 0,1\right]
\times \left[ 0,1\right] ,$ we get%
\begin{eqnarray*}
&& \\
&&\int_{0}^{1}\int_{0}^{1}t^{\alpha -1}k^{\beta -1} \\
&& \\
&&\times \left[ f\left( ta+\left( 1-t\right) b,kc+\left( 1-k\right) d\right)
+f\left( \left( 1-t\right) a+tb,kc+\left( 1-k\right) d\right) \right.  \\
&& \\
&&\left. +f\left( ta+\left( 1-t\right) b,\left( 1-k\right) c+kd\right)
+f\left( \left( 1-t\right) a+tb,\left( 1-k\right) c+kd\right) \right] dkdt \\
&& \\
&\leq &\left[ f\left( a,c\right) +f\left( b,c\right) +f\left( a,d\right)
+f\left( b,d\right) \right]  \\
&& \\
&&\times \left\{ \int_{0}^{1}\int_{0}^{1}t^{\alpha -1}k^{\beta -1}\right.  \\
&& \\
&&\left. \times \left[ h\left( t\right) h\left( k\right) +h\left( t\right)
h\left( 1-k\right) +h\left( 1-t\right) h\left( k\right) +h\left( 1-t\right)
h\left( 1-k\right) \right] dkdt\right\} .
\end{eqnarray*}%
Here, using the change of the variable we have%
\begin{eqnarray*}
&&\left[ h\left( \frac{1}{2}\right) \right] ^{2}\frac{\Gamma \left( \alpha
+1\right) \Gamma \left( \beta +1\right) }{\left( b-a\right) ^{\alpha }\left(
d-c\right) ^{\beta }} \\
&& \\
&&\times \left[ J_{a^{+},c^{+}}^{\alpha ,\beta }f\left( b,d\right)
+J_{a^{+},d^{-}}^{\alpha ,\beta }f\left( b,c\right) +J_{b^{-},c^{+}}^{\alpha
,\beta }f\left( a,d\right) +J_{b^{-},d^{-}}^{\alpha ,\beta }f\left(
a,c\right) \right]  \\
&& \\
&\leq &\left[ h\left( \frac{1}{2}\right) \right] ^{2}\alpha \beta \left[
f\left( a,c\right) +f\left( a,d\right) +f\left( b,c\right) +f\left(
b,d\right) \right]  \\
&& \\
&&\times \left[ \int_{0}^{1}\int_{0}^{1}t^{\alpha -1}k^{\alpha -1}\right.  \\
&& \\
&&\left. \times \left[ h\left( t\right) h\left( k\right) +h\left( t\right)
h\left( 1-k\right) +h\left( 1-t\right) h\left( k\right) +h\left( 1-t\right)
h\left( 1-k\right) \right] dkdt\right] .
\end{eqnarray*}%
The proof is completed.
\end{proof}

\begin{remark}
\label{r1} If we take $h\left( \alpha \right) =\alpha $ in Theorem \ref{t4},
then the inequality (\ref{6}) becomes the \ inequality (\ref{2}) of Theorem %
\ref{t1}.
\end{remark}

\begin{corollary}
\label{c1} If we take $h\left( \alpha \right) =\alpha ^{s}$ in Theorem \ref%
{t1}, we have the following inequality:%
\begin{eqnarray*}
&&f\left( \frac{a+b}{2},\frac{c+d}{2}\right) \\
&& \\
&\leq &\left( \frac{1}{2}\right) ^{2s}\frac{\Gamma \left( \alpha +1\right)
\Gamma \left( \beta +1\right) }{\left( b-a\right) ^{\alpha }\left(
d-c\right) ^{\beta }} \\
&& \\
&&\times \left[ J_{a^{+},c^{+}}^{\alpha ,\beta }f\left( b,d\right)
+J_{a^{+},d^{-}}^{\alpha ,\beta }f\left( b,c\right) +J_{b^{-},c^{+}}^{\alpha
,\beta }f\left( a,d\right) +J_{b^{-},d^{-}}^{\alpha ,\beta }f\left(
a,c\right) \right] \\
&& \\
&\leq &\left( \frac{1}{2}\right) ^{2s}\alpha \beta \left[ f\left( a,c\right)
+f\left( a,d\right) +f\left( b,c\right) +f\left( b,d\right) \right] \\
&& \\
&&\times \left( \frac{1}{\alpha +s}+B\left( \alpha ,s+1\right) \right)
\left( \frac{1}{\beta +s}+B\left( \beta ,s+1\right) \right)
\end{eqnarray*}%
where B is the Beta function,%
\begin{equation*}
B\left( x,y\right) =\int_{0}^{1}t^{x-1}\left( 1-t\right) ^{y-1}dt.
\end{equation*}
\end{corollary}

\begin{theorem}
\label{t5} Let $f:\Delta \subset 
\mathbb{R}
^{2}\rightarrow 
\mathbb{R}
$ be a partial differentiable mapping on $\Delta :=\left[ a,b\right] \times %
\left[ c,d\right] $ in $%
\mathbb{R}
^{2}$ with $0\leq a<b$, $0\leq c<d$. If $\left\vert \frac{\partial ^{2}f}{%
\partial t\partial k}\right\vert $ is a h-convex function on the
co-ordinates on $\Delta $, then one has the inequalities:%
\begin{eqnarray}
&&  \label{9} \\
&&\left\vert \frac{f\left( a,c\right) +f\left( a,d\right) +f\left(
b,c\right) +f\left( b,d\right) }{4}\right.  \notag \\
&&  \notag \\
&&\left. +\left\{ \frac{\Gamma \left( \alpha +1\right) \Gamma \left( \beta
+1\right) }{4\left( b-a\right) ^{\alpha }\left( d-c\right) ^{\beta }}\right.
\right.  \notag \\
&&  \notag \\
&&\left. \left. \times \left[ J_{a^{+},c^{+}}^{\alpha ,\beta }f\left(
b,d\right) +J_{a^{+},d^{-}}^{\alpha ,\beta }f\left( b,c\right)
+J_{b^{-},c^{+}}^{\alpha ,\beta }f\left( a,d\right) +J_{b^{-},d^{-}}^{\alpha
,\beta }f\left( a,c\right) \right] \right\} -A\right\vert  \notag \\
&&  \notag \\
&\leq &\frac{\left( b-a\right) \left( d-c\right) }{4}  \notag \\
&&  \notag \\
&&\times \left\{ \left\vert \frac{\partial ^{2}f}{\partial t\partial k}%
\left( a,c\right) \right\vert \int_{0}^{1}\int_{0}^{1}\left( t^{\alpha
}+\left( 1-t\right) ^{\alpha }\right) \left( k^{\beta }+\left( 1-k\right)
^{\beta }\right) h\left( t\right) h\left( k\right) dkdt\right.  \notag \\
&&  \notag \\
&&+\left\vert \frac{\partial ^{2}f}{\partial t\partial k}\left( b,c\right)
\right\vert \int_{0}^{1}\int_{0}^{1}\left( t^{\alpha }+\left( 1-t\right)
^{\alpha }\right) \left( k^{\beta }+\left( 1-k\right) ^{\beta }\right)
h\left( 1-t\right) h\left( k\right) dkdt  \notag \\
&&  \notag \\
&&+\left\vert \frac{\partial ^{2}f}{\partial t\partial k}\left( a,d\right)
\right\vert \int_{0}^{1}\int_{0}^{1}\left( t^{\alpha }+\left( 1-t\right)
^{\alpha }\right) \left( k^{\beta }+\left( 1-k\right) ^{\beta }\right)
h\left( t\right) h\left( 1-k\right) dkdt  \notag \\
&&  \notag \\
&&\left. +\left\vert \frac{\partial ^{2}f}{\partial t\partial k}\left(
b,d\right) \right\vert \int_{0}^{1}\int_{0}^{1}\left( t^{\alpha }+\left(
1-t\right) ^{\alpha }\right) \left( k^{\beta }+\left( 1-k\right) ^{\beta
}\right) h\left( 1-t\right) h\left( 1-k\right) dkdt\right\}  \notag
\end{eqnarray}%
where%
\begin{eqnarray*}
A &=&\frac{\Gamma \left( \beta +1\right) }{4\left( d-c\right) ^{\beta }}%
\left[ J_{c^{+}}^{\beta }f\left( a,d\right) +J_{c^{+}}^{\beta }f\left(
b,d\right) +J_{d^{-}}^{\beta }f\left( a,c\right) +J_{d^{-}}^{\beta }f\left(
b,c\right) \right] \\
&& \\
&&+\frac{\Gamma \left( \alpha +1\right) }{4\left( b-a\right) ^{\alpha }}%
\left[ J_{a^{+}}^{\alpha }f\left( b,c\right) +J_{a^{+}}^{\alpha }f\left(
b,d\right) +J_{b^{-}}^{\alpha }f\left( a,c\right) +J_{b^{-}}^{\alpha
}f\left( a,d\right) \right] .
\end{eqnarray*}
\end{theorem}

\begin{proof}
From Lemma \ref{l1}, we have%
\begin{eqnarray*}
&&\left\vert \frac{f\left( a,c\right) +f\left( a,d\right) +f\left(
b,c\right) +f\left( b,d\right) }{4}\right.  \\
&& \\
&&+\left\{ \frac{\Gamma \left( \alpha +1\right) \Gamma \left( \beta
+1\right) }{4\left( b-a\right) ^{\alpha }\left( d-c\right) ^{\beta }}\right. 
\\
&& \\
&&\left. \left. \times \left[ J_{a^{+},c^{+}}^{\alpha ,\beta }f\left(
b,d\right) +J_{a^{+},d^{-}}^{\alpha ,\beta }f\left( b,c\right)
+J_{b^{-},c^{+}}^{\alpha ,\beta }f\left( a,d\right) +J_{b^{-},d^{-}}^{\alpha
,\beta }f\left( a,c\right) \right] \right\} -A\right\vert  \\
&& \\
&\leq &\frac{\left( b-a\right) \left( d-c\right) }{4}\left\{
\int_{0}^{1}\int_{0}^{1}t^{\alpha }k^{\beta }\left\vert \frac{\partial ^{2}f%
}{\partial t\partial k}\left( ta+\left( 1-t\right) b,kc+\left( 1-k\right)
d\right) \right\vert dkdt\right.  \\
&& \\
&&+\int_{0}^{1}\int_{0}^{1}\left( 1-t\right) ^{\alpha }k^{\beta }\left\vert 
\frac{\partial ^{2}f}{\partial t\partial k}\left( ta+\left( 1-t\right)
b,kc+\left( 1-k\right) d\right) \right\vert dkdt \\
&& \\
&&+\int_{0}^{1}\int_{0}^{1}t^{\alpha }\left( 1-k\right) ^{\beta }\left\vert 
\frac{\partial ^{2}f}{\partial t\partial k}\left( ta+\left( 1-t\right)
b,kc+\left( 1-k\right) d\right) \right\vert dkdt \\
&& \\
&&\left. -\int_{0}^{1}\int_{0}^{1}\left( 1-t\right) ^{\alpha }\left(
1-k\right) ^{\beta }\left\vert \frac{\partial ^{2}f}{\partial t\partial k}%
\left( ta+\left( 1-t\right) b,kc+\left( 1-k\right) d\right) \right\vert
dkdt\right\} .
\end{eqnarray*}%
Since $\left\vert \frac{\partial ^{2}f}{\partial t\partial k}\right\vert $
is h-convex function on the co-ordinates on $\Delta ,$ then one has:%
\begin{eqnarray*}
&&\left\vert \frac{f\left( a,c\right) +f\left( a,d\right) +f\left(
b,c\right) +f\left( b,d\right) }{4}\right.  \\
&& \\
&&+\left\{ \frac{\Gamma \left( \alpha +1\right) \Gamma \left( \beta
+1\right) }{4\left( b-a\right) ^{\alpha }\left( d-c\right) ^{\beta }}\right. 
\\
&& \\
&&\left. \left. \times \left[ J_{a^{+},c^{+}}^{\alpha ,\beta }f\left(
b,d\right) +J_{a^{+},d^{-}}^{\alpha ,\beta }f\left( b,c\right)
+J_{b^{-},c^{+}}^{\alpha ,\beta }f\left( a,d\right) +J_{b^{-},d^{-}}^{\alpha
,\beta }f\left( a,c\right) \right] \right\} -A\right\vert 
\end{eqnarray*}%
\begin{eqnarray*}
&\leq &\frac{\left( b-a\right) \left( d-c\right) }{4} \\
&& \\
&&\times \left\{ \left\vert \frac{\partial ^{2}f}{\partial t\partial k}%
\left( a,c\right) \right\vert \int_{0}^{1}\int_{0}^{1}\left( t^{\alpha
}+\left( 1-t\right) ^{\alpha }\right) \left( k^{\beta }+\left( 1-k\right)
^{\beta }\right) h\left( t\right) h\left( k\right) dkdt\right.  \\
&& \\
&&+\left\vert \frac{\partial ^{2}f}{\partial t\partial k}\left( b,c\right)
\right\vert \int_{0}^{1}\int_{0}^{1}\left( t^{\alpha }+\left( 1-t\right)
^{\alpha }\right) \left( k^{\beta }+\left( 1-k\right) ^{\beta }\right)
h\left( 1-t\right) h\left( k\right) dkdt \\
&& \\
&&+\left\vert \frac{\partial ^{2}f}{\partial t\partial k}\left( a,d\right)
\right\vert \int_{0}^{1}\int_{0}^{1}\left( t^{\alpha }+\left( 1-t\right)
^{\alpha }\right) \left( k^{\beta }+\left( 1-k\right) ^{\beta }\right)
h\left( t\right) h\left( 1-k\right) dkdt \\
&& \\
&&\left. +\left\vert \frac{\partial ^{2}f}{\partial t\partial k}\left(
b,d\right) \right\vert \int_{0}^{1}\int_{0}^{1}\left( t^{\alpha }+\left(
1-t\right) ^{\alpha }\right) \left( k^{\beta }+\left( 1-k\right) ^{\beta
}\right) h\left( 1-t\right) h\left( 1-k\right) dkdt\right\} .
\end{eqnarray*}

The proof is completed.
\end{proof}

\begin{remark}
\label{r2} If we take $h\left( \alpha \right) =\alpha $ in Theorem \ref{t5},
then the inequality (\ref{9}) becomes the \ inequality (\ref{3}) of Theorem %
\ref{t2}.
\end{remark}

\begin{corollary}
\label{c2} If we take $h\left( \alpha \right) =\alpha ^{s}$ in Theorem \ref%
{t5}, we have%
\begin{eqnarray*}
&&\left\vert \frac{f\left( a,c\right) +f\left( a,d\right) +f\left(
b,c\right) +f\left( b,d\right) }{4}\right. \\
&& \\
&&+\left\{ \frac{\Gamma \left( \alpha +1\right) \Gamma \left( \beta
+1\right) }{4\left( b-a\right) ^{\alpha }\left( d-c\right) ^{\beta }}\right.
\\
&& \\
&&\left. \left. \left[ J_{a^{+},c^{+}}^{\alpha ,\beta }f\left( b,d\right)
+J_{a^{+},d^{-}}^{\alpha ,\beta }f\left( b,c\right) +J_{b^{-},c^{+}}^{\alpha
,\beta }f\left( a,d\right) +J_{b^{-},d^{-}}^{\alpha ,\beta }f\left(
a,c\right) \right] \right\} -A\right\vert \\
&& \\
&\leq &\frac{\left( b-a\right) \left( d-c\right) }{4} \\
&& \\
&&\times \left[ \left\vert \frac{\partial ^{2}f}{\partial t\partial k}\left(
a,c\right) \right\vert +\left\vert \frac{\partial ^{2}f}{\partial t\partial k%
}\left( b,c\right) \right\vert +\left\vert \frac{\partial ^{2}f}{\partial
t\partial k}\left( a,d\right) \right\vert +\left\vert \frac{\partial ^{2}f}{%
\partial t\partial k}\left( b,d\right) \right\vert \right] \\
&& \\
&&\times \left( \frac{1}{\alpha +s+1}+B\left( s+1,\alpha +1\right) \right)
\left( \frac{1}{\beta +s+1}+B\left( s+1,\beta +1\right) \right)
\end{eqnarray*}%
where%
\begin{eqnarray*}
A &=&\frac{\Gamma \left( \beta +1\right) }{4\left( d-c\right) ^{\beta }}%
\left[ J_{c^{+}}^{\beta }f\left( a,d\right) +J_{c^{+}}^{\beta }f\left(
b,d\right) +J_{d^{-}}^{\beta }f\left( a,c\right) +J_{d^{-}}^{\beta }f\left(
b,c\right) \right] \\
&& \\
&&+\frac{\Gamma \left( \alpha +1\right) }{4\left( b-a\right) ^{\alpha }}%
\left[ J_{a^{+}}^{\alpha }f\left( b,c\right) +J_{a^{+}}^{\alpha }f\left(
b,d\right) +J_{b^{-}}^{\alpha }f\left( a,c\right) +J_{b^{-}}^{\alpha
}f\left( a,d\right) \right]
\end{eqnarray*}%
and B is the Beta function,%
\begin{equation*}
B\left( x,y\right) =\int_{0}^{1}t^{x-1}\left( 1-t\right) ^{y-1}dt.
\end{equation*}
\end{corollary}

\begin{theorem}
\label{t6} Let $f:\Delta \subset 
\mathbb{R}
^{2}\rightarrow 
\mathbb{R}
$ be a partial differentiable mapping on $\Delta :=\left[ a,b\right] \times %
\left[ c,d\right] $ in $%
\mathbb{R}
^{2}$ with $0\leq a<b$, $0\leq c<d$. If $\left\vert \frac{\partial ^{2}f}{%
\partial t\partial k}\right\vert ^{q},q>1,$ is a h-convex function on the
co-ordinates on $\Delta $, then one has the inequalities:%
\begin{eqnarray}
&&  \label{10} \\
&&\left\vert \frac{f\left( a,c\right) +f\left( a,d\right) +f\left(
b,c\right) +f\left( b,d\right) }{4}\right.   \notag \\
&&  \notag \\
&&+\left\{ \frac{\Gamma \left( \alpha +1\right) \Gamma \left( \beta
+1\right) }{4\left( b-a\right) ^{\alpha }\left( d-c\right) ^{\beta }}\right. 
\notag \\
&&  \notag \\
&&\left. \left. \times \left[ J_{a^{+},c^{+}}^{\alpha ,\beta }f\left(
b,d\right) +J_{a^{+},d^{-}}^{\alpha ,\beta }f\left( b,c\right)
+J_{b^{-},c^{+}}^{\alpha ,\beta }f\left( a,d\right) +J_{b^{-},d^{-}}^{\alpha
,\beta }f\left( a,c\right) \right] \right\} -A\right\vert   \notag \\
&&  \notag \\
&\leq &\frac{\left( b-a\right) \left( d-c\right) }{\left[ \left( \alpha
p+1\right) \left( \beta p+1\right) \right] ^{\frac{1}{p}}}  \notag \\
&&  \notag \\
&&\times \left( \left\vert \frac{\partial ^{2}f}{\partial k\partial t}\left(
a,c\right) \right\vert ^{q}\int_{0}^{1}\int_{0}^{1}h\left( t\right) h\left(
k\right) dkdt\right.   \notag \\
&&  \notag \\
&&+\left\vert \frac{\partial ^{2}f}{\partial k\partial t}\left( a,d\right)
\right\vert ^{q}\int_{0}^{1}\int_{0}^{1}h\left( t\right) h\left( 1-k\right)
dkdt  \notag \\
&&  \notag \\
&&+\left\vert \frac{\partial ^{2}f}{\partial k\partial t}\left( b,c\right)
\right\vert ^{q}\int_{0}^{1}\int_{0}^{1}h\left( 1-t\right) h\left( k\right)
dkdt  \notag \\
&&  \notag \\
&&\left. +\left\vert \frac{\partial ^{2}f}{\partial k\partial t}\left(
b,d\right) \right\vert ^{q}\int_{0}^{1}\int_{0}^{1}h\left( 1-t\right)
h\left( 1-k\right) dkdt\right) ^{\frac{1}{q}}  \notag
\end{eqnarray}%
where%
\begin{eqnarray*}
A &=&\frac{\Gamma \left( \beta +1\right) }{4\left( d-c\right) ^{\beta }}%
\left[ J_{c^{+}}^{\beta }f\left( a,d\right) +J_{c^{+}}^{\beta }f\left(
b,d\right) +J_{d^{-}}^{\beta }f\left( a,c\right) +J_{d^{-}}^{\beta }f\left(
b,c\right) \right]  \\
&& \\
&&+\frac{\Gamma \left( \alpha +1\right) }{4\left( b-a\right) ^{\alpha }}%
\left[ J_{a^{+}}^{\alpha }f\left( b,c\right) +J_{a^{+}}^{\alpha }f\left(
b,d\right) +J_{b^{-}}^{\alpha }f\left( a,c\right) +J_{b^{-}}^{\alpha
}f\left( a,d\right) \right] 
\end{eqnarray*}%
and $\frac{1}{p}+\frac{1}{q}=1.$
\end{theorem}

\begin{proof}
From Lemma \ref{l1}, we have%
\begin{eqnarray*}
&&\left\vert \frac{f\left( a,c\right) +f\left( a,d\right) +f\left(
b,c\right) +f\left( b,d\right) }{4}\right.  \\
&& \\
&&+\left\{ \frac{\Gamma \left( \alpha +1\right) \Gamma \left( \beta
+1\right) }{4\left( b-a\right) ^{\alpha }\left( d-c\right) ^{\beta }}\right. 
\\
&& \\
&&\left. \left. \times \left[ J_{a^{+},c^{+}}^{\alpha ,\beta }f\left(
b,d\right) +J_{a^{+},d^{-}}^{\alpha ,\beta }f\left( b,c\right)
+J_{b^{-},c^{+}}^{\alpha ,\beta }f\left( a,d\right) +J_{b^{-},d^{-}}^{\alpha
,\beta }f\left( a,c\right) \right] \right\} -A\right\vert 
\end{eqnarray*}%
\begin{eqnarray*}
&\leq &\frac{\left( b-a\right) \left( d-c\right) }{4}\left\{
\int_{0}^{1}\int_{0}^{1}t^{\alpha }k^{\beta }\left\vert \frac{\partial ^{2}f%
}{\partial t\partial k}\left( ta+\left( 1-t\right) b,kc+\left( 1-k\right)
d\right) \right\vert dkdt\right.  \\
&& \\
&&+\int_{0}^{1}\int_{0}^{1}\left( 1-t\right) ^{\alpha }k^{\beta }\left\vert 
\frac{\partial ^{2}f}{\partial t\partial k}\left( ta+\left( 1-t\right)
b,kc+\left( 1-k\right) d\right) \right\vert dkdt \\
&& \\
&&+\int_{0}^{1}\int_{0}^{1}t^{\alpha }\left( 1-k\right) ^{\beta }\left\vert 
\frac{\partial ^{2}f}{\partial t\partial k}\left( ta+\left( 1-t\right)
b,kc+\left( 1-k\right) d\right) \right\vert dkdt \\
&& \\
&&\left. -\int_{0}^{1}\int_{0}^{1}\left( 1-t\right) ^{\alpha }\left(
1-k\right) ^{\beta }\left\vert \frac{\partial ^{2}f}{\partial t\partial k}%
\left( ta+\left( 1-t\right) b,kc+\left( 1-k\right) d\right) \right\vert
dkdt\right\} .
\end{eqnarray*}%
By using the well known H\"{o}lder's inequality for double integrals, we get%
\begin{eqnarray*}
&&\left\vert \frac{f\left( a,c\right) +f\left( a,d\right) +f\left(
b,c\right) +f\left( b,d\right) }{4}\right.  \\
&& \\
&&+\left\{ \frac{\Gamma \left( \alpha +1\right) \Gamma \left( \beta
+1\right) }{4\left( b-a\right) ^{\alpha }\left( d-c\right) ^{\beta }}\right. 
\\
&& \\
&&\left. \left. \times \left[ J_{a^{+},c^{+}}^{\alpha ,\beta }f\left(
b,d\right) +J_{a^{+},d^{-}}^{\alpha ,\beta }f\left( b,c\right)
+J_{b^{-},c^{+}}^{\alpha ,\beta }f\left( a,d\right) +J_{b^{-},d^{-}}^{\alpha
,\beta }f\left( a,c\right) \right] \right\} -A\right\vert  \\
&& \\
&\leq &\frac{\left( b-a\right) \left( d-c\right) }{4}\left\{ \left(
\int_{0}^{1}\int_{0}^{1}t^{p\alpha }k^{p\beta }dkdt\right) ^{\frac{1}{p}%
}+\left( \int_{0}^{1}\int_{0}^{1}\left( 1-t\right) ^{p\alpha }k^{p\beta
}dkdt\right) ^{\frac{1}{p}}\right.  \\
&& \\
&&\left. \left( \int_{0}^{1}\int_{0}^{1}t^{p\alpha }\left( 1-k\right)
^{p\beta }dkdt\right) ^{\frac{1}{p}}+\left( \int_{0}^{1}\int_{0}^{1}\left(
1-t\right) ^{p\alpha }\left( 1-k\right) ^{p\beta }dkdt\right) ^{\frac{1}{p}%
}\right\}  \\
&& \\
&&\left( \int_{0}^{1}\int_{0}^{1}\left\vert \frac{\partial ^{2}f}{\partial
t\partial k}\left( ta+\left( 1-t\right) b,kc+\left( 1-k\right) d\right)
\right\vert ^{q}dkdt\right) ^{^{\frac{1}{q}}}
\end{eqnarray*}%
Since $\left\vert \frac{\partial ^{2}f}{\partial t\partial k}\right\vert ^{q}
$ is h-convex function on the co-ordinates on $\Delta ,$ then one has:%
\begin{eqnarray*}
&&\left\vert \frac{f\left( a,c\right) +f\left( a,d\right) +f\left(
b,c\right) +f\left( b,d\right) }{4}\right.  \\
&& \\
&&+\left\{ \frac{\Gamma \left( \alpha +1\right) \Gamma \left( \beta
+1\right) }{4\left( b-a\right) ^{\alpha }\left( d-c\right) ^{\beta }}\right. 
\\
&& \\
&&\left. \left. \times \left[ J_{a^{+},c^{+}}^{\alpha ,\beta }f\left(
b,d\right) +J_{a^{+},d^{-}}^{\alpha ,\beta }f\left( b,c\right)
+J_{b^{-},c^{+}}^{\alpha ,\beta }f\left( a,d\right) +J_{b^{-},d^{-}}^{\alpha
,\beta }f\left( a,c\right) \right] \right\} -A\right\vert 
\end{eqnarray*}%
\begin{eqnarray*}
&\leq &\frac{\left( b-a\right) \left( d-c\right) }{\left[ \left( \alpha
p+1\right) \left( \beta p+1\right) \right] ^{\frac{1}{p}}} \\
&& \\
&&\times \left( \left\vert \frac{\partial ^{2}f}{\partial k\partial t}\left(
a,c\right) \right\vert ^{q}\int_{0}^{1}\int_{0}^{1}h\left( t\right) h\left(
k\right) dkdt\right.  \\
&& \\
&&+\left\vert \frac{\partial ^{2}f}{\partial k\partial t}\left( a,d\right)
\right\vert ^{q}\int_{0}^{1}\int_{0}^{1}h\left( t\right) h\left( 1-k\right)
dkdt \\
&& \\
&&+\left\vert \frac{\partial ^{2}f}{\partial k\partial t}\left( b,c\right)
\right\vert ^{q}\int_{0}^{1}\int_{0}^{1}h\left( 1-t\right) h\left( k\right)
dkdt \\
&& \\
&&\left. +\left\vert \frac{\partial ^{2}f}{\partial k\partial t}\left(
b,d\right) \right\vert ^{q}\int_{0}^{1}\int_{0}^{1}h\left( 1-t\right)
h\left( 1-k\right) dkdt\right) ^{\frac{1}{q}}
\end{eqnarray*}

and the proof is completed.
\end{proof}

\begin{remark}
\label{r3} If we take $h\left( \alpha \right) =\alpha $ in Theorem \ref{t6},
then the inequality (\ref{10}) becomes the \ inequality (\ref{4}) of Theorem %
\ref{t3}.
\end{remark}

\begin{corollary}
\label{c3} If we take $h\left( \alpha \right) =\alpha ^{s}$ in Theorem \ref%
{t5}, we have%
\begin{eqnarray*}
&&\left\vert \frac{f\left( a,c\right) +f\left( a,d\right) +f\left(
b,c\right) +f\left( b,d\right) }{4}\right. \\
&& \\
&&+\left\{ \frac{\Gamma \left( \alpha +1\right) \Gamma \left( \beta
+1\right) }{4\left( b-a\right) ^{\alpha }\left( d-c\right) ^{\beta }}\right.
\\
&& \\
&&\left. \left. \times \left[ J_{a^{+},c^{+}}^{\alpha ,\beta }f\left(
b,d\right) +J_{a^{+},d^{-}}^{\alpha ,\beta }f\left( b,c\right)
+J_{b^{-},c^{+}}^{\alpha ,\beta }f\left( a,d\right) +J_{b^{-},d^{-}}^{\alpha
,\beta }f\left( a,c\right) \right] \right\} -A\right\vert \\
&& \\
&\leq &\frac{\left( b-a\right) \left( d-c\right) }{\left[ \left( \alpha
p+1\right) \left( \beta p+1\right) \right] ^{\frac{1}{p}}} \\
&& \\
&&\times \left( \frac{1}{\left( s+1\right) ^{2}}\left[ \left\vert \frac{%
\partial ^{2}f}{\partial k\partial t}\left( a,c\right) \right\vert
^{q}+\left\vert \frac{\partial ^{2}f}{\partial k\partial t}\left( a,d\right)
\right\vert ^{q}+\left\vert \frac{\partial ^{2}f}{\partial k\partial t}%
\left( b,c\right) \right\vert ^{q}+\left\vert \frac{\partial ^{2}f}{\partial
k\partial t}\left( b,d\right) \right\vert ^{q}\right] \right) ^{\frac{1}{q}}
\end{eqnarray*}%
where%
\begin{eqnarray*}
A &=&\frac{\Gamma \left( \beta +1\right) }{4\left( d-c\right) ^{\beta }}%
\left[ J_{c^{+}}^{\beta }f\left( a,d\right) +J_{c^{+}}^{\beta }f\left(
b,d\right) +J_{d^{-}}^{\beta }f\left( a,c\right) +J_{d^{-}}^{\beta }f\left(
b,c\right) \right] \\
&& \\
&&+\frac{\Gamma \left( \alpha +1\right) }{4\left( b-a\right) ^{\alpha }}%
\left[ J_{a^{+}}^{\alpha }f\left( b,c\right) +J_{a^{+}}^{\alpha }f\left(
b,d\right) +J_{b^{-}}^{\alpha }f\left( a,c\right) +J_{b^{-}}^{\alpha
}f\left( a,d\right) \right] .
\end{eqnarray*}
\end{corollary}

\end{document}